\documentclass[12pt]{amsart}
\usepackage[utf8]{inputenc}
\usepackage[OT1]{fontenc}

\usepackage{hyperref}
\usepackage{amssymb}
\usepackage{amsmath}
\usepackage{cite}

\usepackage{graphics}
\usepackage{amsthm}
\newtheorem{theorem}{Theorem}[section]
\newtheorem{proposition}[theorem]{Proposition}
\newtheorem{corollary}[theorem]{Corollary}
\newtheorem{lemma}[theorem]{Lemma}

\newtheorem{definition}[theorem]{Definition}

\newtheorem{remark}[theorem]{Remark}

\newtheorem{innercustomthm}{}
\newenvironment{customthm}[1]
  {\renewcommand\theinnercustomthm{#1}\innercustomthm}
  {\endinnercustomthm}

\numberwithin{equation}{section}

\usepackage{tagging}
\usepackage{mathrsfs}

\usepackage{amsmath}
\usepackage{esint}
\usepackage{mathtools}
\usepackage{amssymb}
\usepackage{multicol}
\usepackage{color}
\usepackage[makeroom]{cancel} 
\usepackage{multicol}
\usepackage{upgreek}
\usepackage[top=2cm, bottom=3cm, left=2cm, right=3cm]{geometry}
\usepackage{graphicx}
{}

\usepackage{makeidx}
\usepackage{tikz}
\usepackage{tikz-cd}
\usepackage{pgfplots}
\graphicspath{ {images/} }
\graphicspath{{/texfiles/}}
\usepackage[english]{babel}
\usepackage{hyphenat}
\pagecolor{white}
\usepackage{fancyhdr}
\usepackage{appendix}
\usepackage{ tipa }
\usepackage{listings}
\usepackage{color}
\usepackage{graphicx}

\DeclareMathOperator{\df}{d\textit{f}}

\renewcommand{\leq}{\leqslant}
\renewcommand{\geq}{\geqslant}
\title[Modulus of Continuity of Solutions on Stein Spaces]{Modulus of Continuity of Solutions to Complex Monge-Ampère Equations on Stein Spaces}
\author{Guilherme Cerqueira-Gonçalves}
\address{Institut de Mathématiques de Toulouse; UMR 5219, Université de Toulouse;
CNRS, UPS, 118 route de Narbonne, F-31062 Toulouse Cedex 9, France}
\email{guicerqg@gmail.com / guilherme.cerqueira\mathunderscore{}goncalves@math.univ-toulouse.fr}
\urladdr{\href{https://guicerqmath.github.io/}{https://guicerqmath.github.io/}}
\date{10 November 2024}

\pgfplotsset{compat=1.18}
\begin{document}

\begin{abstract}
    In this paper, we study the modulus of continuity of 
    solutions to Dirichlet problems for complex Monge-Ampère 
    equations with $L^p$ densities on Stein spaces with 
    isolated singularities.
    In particular, we prove such solutions are Hölder 
    continuous outside singular points
    if the boundary data is Hölder continuous. 
\end{abstract}

\maketitle

\tableofcontents

\section*{Introduction}\label{Sing: background}

Let $X$ be a complex space that is reduced, locally 
irreducible, and of complex dimension 
$n \geqslant 1$ with an isolated singularity, $X_{\text{sing}}$. 
Equip $X$ with a Hermitian metric whose fundamental 
form is $\beta$, a positive $(1,1)$-form, and 
let $d_{\beta}(\cdot, \cdot)$ be the distance 
induced by $\beta$. 
Let $\Omega$ be a bounded, strongly pseudoconvex open 
subset of $X$, such that $X_{\text{sing}} \subset \Omega$,
and fix $\rho$ as a smooth, strictly 
plurisubharmonic defining function, i.e., $\Omega = \{\rho < 0\}$, 
hence $\Omega$ is a Stein space. Given $\phi \in C^0(\partial \Omega)$ 
and $f \in L^p(\Omega, \beta^n)$ with $p > 1$, 
we consider the Dirichlet problem:

\noindent
\[
MA(\Omega, \phi, f):
\begin{cases}
    {(dd^c u)}^n = f \beta^n & \text{in} \quad \Omega,  \\
    u = \phi & \text{on} \quad \partial \Omega,
\end{cases}
\]

\noindent where $d = \partial + \overline{\partial}$ and $d^{c} = i(\overline{\partial} - \partial)$.
This Dirichlet problem for the complex Monge-Ampère 
operator has been intensively 
studied in domains of $\mathbb{C}^n$ (see~\cite{GZ17} for a historical account).

The study of complex Monge-Ampère equations is key 
to understanding canonical metrics in Kähler manifolds. 
This significance extends to mildly singular 
varieties as well. 
Interactions with birational geometry, such as the 
Minimal Model Program, make it enticing across 
fields to comprehend these equations.

In this paper, our objective is to explore the impact of 
boundary data and singularity types on the modulus of 
continuity of solutions. In particular, we prove that if 
$\phi$ is \textit{Hölder continuous}, then so is 
the solution of $MA(\Omega, \phi, f)$ outside the singular 
point.

Regarding applications of the Hölder regularity of solutions to 
complex Monge-Ampère equations, we refer to~\cite{GGZ23c} 
for some geometric consequences 
of the Hölder continuity of Kähler-Einstein potentials and 
to~\cite{DNS10} for more details about applications in 
complex dynamics.

Several works have focused on the Hölder continuity of solutions, 
both on domains of $\mathbb{C}^{n}$ (see~\cite{GKZ08,Cha14,BKPZ16,Cha15}) 
and on compact complex manifolds 
(see~\cite{DDGHKZ14,KN18,LPT21,DKN22}).

However, there are only a few literature dealing 
with this problem on domains with singularities.
Recently, it was proved in~\cite{GGZ23a} that the 
solution $u = u(\Omega, \phi, f)$ for 
$MA(\Omega, \phi, f)$ is continuous and unique in $\overline{\Omega}$. 
Further smoothness properties have been provided in~\cite{DFS23,F23,GT23} 
on the regular part of domains that contain an isolated singularity in the 
case of smooth boundary data.

Our main result is the following:

\begin{customthm}{Theorem A}\label{teo: main result}

Assume that $\phi \in C^0(\partial \Omega)$ and $f \in L^p(\Omega, \beta^n)$ with $p > 1$. Then the unique solution 
$u \in PSH(\Omega) \cap C^0(\overline{\Omega})$ to $MA(\Omega, \phi, f)$ 
has the following modulus of continuity at $x \in \Omega$:

\noindent
\[
\omega_{u, x}(t) \leqslant C_x \max\{\omega_{\phi}(t^{1/2}), t^{\frac{1}{nq+1}}\}
\]

\noindent for some constant $C_x > 0$ such that $C_x \to +\infty$ as $x \to X_{\text{sing}}$, where $\omega_{u, x}$ is the modulus 
of continuity of $u$ at the point $x$ and $\frac{1}{p} + \frac{1}{q} = 1$. 
\end{customthm}

In particular,~\ref{teo: main result} implies:

\begin{customthm}{Corollary B}\label{teo: holder sing}
When $\phi \in C^{0, \alpha}(\partial \Omega)$, $0 < \alpha \leqslant 1$, and $f \in L^p(\Omega, \beta^n)$ with $p > 1$, the 
unique solution $u \in PSH(\Omega) \cap C^0(\overline{\Omega})$ 
to $MA(\Omega, \phi, f)$ is 
$\alpha^*$-\textit{Hölder continuous} outside the singular point, 
for $\alpha^* < \min\{\frac{\alpha}{2}, \frac{1}{nq+1}\}$. 
\end{customthm}

\begin{customthm}{Remark}
This Hölder exponent\footnote{Relative to the 
induced distance of the chosen $\beta$.}
is the same for smooth domains~\cite[Theorem 1.2]{Cha15}.
\end{customthm}

\begin{customthm}{Remark}
    One can similarly treat the case when $X$ has finitely 
    many isolated singularities.
\end{customthm}

\subsection*{Acknowledgements}

The author would like to first thank his PhD advisor, 
Vincent Guedj, for his support, guidance and comments. 
The author also thanks Ahmed Zeriahi and Chung-Ming Pan 
for useful discussions and references. The author would
like to thank the referee for useful comments and corrections. This work 
received support from the University Research School 
EUR-MINT
(State support managed by the National Research 
Agency for Future Investments
program bearing the reference ANR-18-EURE-0023).

\section{Preliminaries}\label{Sing: preliminaries}

\subsection{Complex analysis on Stein spaces}

Throughout this paper, we let $X$ be a 
reduced, locally irreducible complex 
analytic space of pure dimension $n \geqslant 1$.
 We denote by $X_{\text{reg}}$ the 
 complex manifold of regular points of $X$ and 
$
X_{\text{sing}}:=X \setminus X_{\text{reg}}
$
the set of singular points, which is an analytic subset 
of $X$ with complex codimension $\geqslant 1$.

By definition, for each point $x_{0} \in X$, there exists 
a neighborhood $U$ of $x_{0}$ and a local embedding 
$j: U \hookrightarrow \mathbb{C}^N$ onto an analytic 
subset of $\mathbb{C}^N$ for some $N \geqslant 1$.
These local embeddings allow us to define the 
spaces of smooth forms of given degree on $X$ as smooth 
forms on $X_{\text{reg}}$ that are locally on $X$  
restrictions of an ambient form on $\mathbb{C}^N$. 
Other differential notions and operators, such as 
holomorphic and plurisubharmonic functions, can also 
be defined in this way. (See~\cite{Dem85} for more details).
Two different notions can be defined for 
plurisubharmonicity:

\begin{definition}\label{def: psh}
Let $u: X \to \mathbb{R} \cup\{-\infty\}$ 
be a given function.

\begin{enumerate}
    \item We say that $u$ is plurisubharmonic (psh for short) on $X$ 
    if it is locally the restriction of a 
    plurisubharmonic function in a local embedding 
    of $X$.

    \item We say that $u$ is weakly plurisubharmonic 
    on $X$ if $u$ is locally bounded from above on 
    $X$, and its restriction to the complex manifold 
    $X_{\mathrm{reg}}$ is plurisubharmonic.

\end{enumerate}

\end{definition}

Forn{\ae}ss and Narasimhan~\cite[Theorem 5.3.1]{FN80} proved 
that $u$ is plurisubarmonic on $X$ if and only if, 
for any analytic disc 
$h: \mathbb{D} \to X$, the function 
$u \circ h$ is subharmonic or 
identically $-\infty$.
If $u$ is weakly plurisubharmonic on $X$, then $u$ 
is plurisubharmonic on $X_{\text{reg}}$, and thus 
upper semi-continuous on $X_{\text{reg}}$. 
It is natural to extend $u$ to $X$ using the following formula:

\noindent
\begin{equation}\label{eq: ext sing pts}
    u^*(x):=\limsup _{X_{\mathrm{reg}} \ni y \rightarrow x} u(y), \quad x \in X.
\end{equation}

The function $u^*$ is upper semi-continuous, 
locally integrable on $X$, and satisfies 
$dd^c u^* \geqslant$ 0 in the sense of 
currents on $X$~\cite[Théorème 1.7]{Dem85}. The two 
notions are equivalent when $X$ is locally 
irreducible, as shown in~\cite[Théorèm 1.10]{Dem85}:

\begin{theorem} Let $X$ be a 
    locally irreducible analytic space and 
    $u: X \to $ $\mathbb{R} \cup\{-\infty\}$ 
    be a weakly plurisubharmonic function on $X$. Then 
    the function $u^*$ defined by \eqref{eq: ext sing pts} is 
    psh on $X$.
    
\end{theorem}

Note that since $u$ is plurisubharmonic on 
$X_{\text{reg}}$, we have $u^*=u$ on 
$X_{\text{reg}}$. Then $u^*$ is the 
upper semi-continuous extension of 
$\left.u\right|_{X_{\mathrm{reg}}}$ to $X$. 
With this, we have the following:

\begin{corollary}
Let $\mathcal{U} \subset \mathrm{PSH}(X)$ be a 
non-empty family of plurisubharmonic functions 
which is locally bounded from above on $X$. 
Then its upper envelope

\noindent
$$
{U}:=\sup \{u ; u \in \mathcal{U}\}
$$

\noindent is a well-defined Borel function whose 
upper semi-continuous regularization $U^*$ is 
psh on $X$.
\end{corollary}

Recall that from~\cite[Theorem 6.1]{FN80} $X$ is 
Stein if it admits a $C^2$ strongly 
plurisubharmonic exhaustion. 
We will use the following definition:

\begin{definition}
    A domain $\Omega \Subset X$ is 
    strongly pseudoconvex if it admits a negative 
    smooth, strongly plurisubharmonic 
    defining function, i.e., a strongly 
    plurisubharmonic function $\rho$ in a neighborhood $\Omega^{\prime}$ 
    of $\overline{\Omega}$ such that 
    $\Omega:=\left\{x \in \Omega^{\prime} ; \rho(x)<0\right\}$ 
    and for any $c<0$,

\noindent
    $$
    \Omega_{c}:=\left\{x \in \Omega^{\prime} ; \rho(x)<c\right\} \Subset \Omega
    $$
    
    \noindent is relatively compact.
    
\end{definition}

We denote by $\operatorname{PSH}(X)$ the set of 
plurisubharmonic functions on $X$.

On complex spaces, the complex Monge-Ampère operator has been 
defined and studied in~\cite{Bed82} and~\cite{Dem85}. 
In this setting, if $u \in \operatorname{PSH}(X) \cap L_{\text{loc}}^{\infty}(X)$, 
the Monge-Ampère measure $\left(dd^c u\right)^n$ is 
well defined on the regular part $X_{\text{reg}}$ and can be extended 
to $X$ as a Borel measure with zero mass on $X_{\text{sing}}$. 
This notion extends the foundational work of 
pluripotential theory by Bedford-Taylor~\cite{BT76,BT82}.

Consequently, several standard properties of the complex Monge-Ampère 
operator acting on 
$\operatorname{PSH}(X) \cap L_{\text{loc}}^{\infty}(X)$ 
extend to this setting (see~\cite{Bed82, Dem85}). 
In particular, we have the following comparison principle~\cite[Theorem 4.3]{Bed82}:

\begin{proposition}[Comparison principle]\label{prop: comp prin}
    Let $\Omega \Subset X$ be a relatively 
compact open set and 
$u, v \in \operatorname{PSH}(\Omega) \cap L^{\infty}(\Omega)$. 
Assume that 
$\liminf _{x \rightarrow \zeta}(u(x)-v(x)) \geqslant 0$ 
for any $\zeta \in \partial \Omega$. 
Then

\noindent
$$
\int_{\{u<v\}}\left(dd^c v\right)^n \leqslant \int_{\{u<v\}}\left(dd^c u\right)^n.
$$

In particular, if 
$\left(dd^c u\right)^n \leqslant\left(dd^c v\right)^n$ 
weakly on $\Omega$, then $v \leqslant u$ on $\Omega$.

\end{proposition}

\subsection{Notations}\label{Notations}

We fix the following notations throughout this paper:

\begin{itemize}
    \item $X$ is a complex space that is reduced and locally 
     irreducible of complex dimension $n \geq 1$ with an 
     isolated singularity $X_{\text{sing}}$.

    \item $\beta$ is a fixed smooth positive (1,1)-form on $X$, the fundamental form 
    of a Hermitian metric on $X$.
    
    \item $d_{\tau}$ is the distance induced by a 
    fundamental form $\tau$ on some manifold $(Y,\tau)$. 
    When no confusion can occur $d_{\beta}$ will be 
    written as $d(\cdot, \cdot)$. Also, for 
    any $x \in X$ we set $B_{r}(x) \Subset X$ as the ball 
    of radius $r$ centered in $x$ for the specified metric,
    if no metric is specified then the reference is $d_{\beta}$.

    \item $\Omega \Subset X$ is a bounded, strongly pseudoconvex 
    open subset of $X$ and fix $\rho$ smooth
    strictly psh defining function, i.e., $\Omega = \{ \rho <0 \}$,
    that will be the Stein space where we work on.

    \item $0 \leqslant f \in L^p(\Omega, \beta^n)$, $p>1$ and $1/p + 1/q =1$ and $\phi \in C^0(\partial \Omega)$.

    \item $u(\Omega, \phi, f)$ is the solution to the 
    Dirichlet problem:
    \[
    MA(\Omega,\phi,f): \begin{cases}
    {(dd^c u)}^n = f \beta^n & \text{ in} \quad \Omega,  \\
    u = \phi & \text{ on} \quad \partial \Omega.
    \end{cases}
    \]

    \item $\omega_{g,x}$ is the modulus of continuity of 
    some function $g$ at the point $x$.

    \item $\lambda(x) = d(x,X_{\text{sing}})$ the distance of a point $x$
    to the singular point. The point $x$ will be omitted 
    when there is no chance of confusion.

    \item  $\mathbb{B}_{2n}$ is the volume of unit ball in $\mathbb{C}^n$. $dV$ is the standard euclidian volume form in $\mathbb{C}^n$. $d_{\mathbb{C}^n}$ is the standard euclidian 
    distance in $\mathbb{C}^n$.
\end{itemize}

\subsection{Useful Results}

We need the stability estimate from~\cite[Proposition 1.8]{GGZ23a}:

\begin{theorem}\label{teo: stability sing}
Let $\varphi, \psi$ be two bounded plurisubharmonic functions 
in $\Omega$ and $(dd^c \varphi)^n = f\beta^n$ in $\Omega$. 
Fix $0 \leqslant \gamma < \frac{1}{nq+1}$. 
Then there exists a uniform constant $C = C(\gamma, ||f||_{L^p (\Omega)})>0$ such that:

\noindent
\[
\underset{\Omega}{\sup}(\psi - \varphi) \leqslant 
\underset{\partial \Omega}{\sup}~(\psi - \varphi)_+^*  + 
C ( ||(\psi - \varphi)_+ ||_{L^1 (\Omega,\beta^n)})^\gamma
\]

\noindent where $(\psi - \varphi)_+ := \max(\psi - \varphi, 0)$ 
and $w^*$ is the upper semi-continuous extension of 
a bounded function $w$ on $\Omega$ to the boundary, 
i.e., $w^*(\xi) := \limsup_{z \to \xi}w(z)$.

\end{theorem}

The local theory~\cite[Theorems 3.3, 3.4 and Lemma 3.5]{BKPZ16} 
will be used in the proof of Lemma~\ref{lema: ctrl term 2}. Here reformulated to:

\begin{theorem}\label{teo: local estimate sing}

Let $\widehat{\Omega} \Subset \mathbb{C}^n$ 
strongly pseudoconvex domain, 
$0 \leqslant \widehat{f} \in L^p(\widehat{\Omega}); p>1$ 
and $\varepsilon > 0$ small enough. 
For $\widehat{u} \in PSH(\widehat{\Omega})\cap C^0(\overline{\widehat{\Omega}})$ 
such that $(dd^c\widehat{u})^n = \widehat{f}dV$ 
in $\widehat{\Omega}$. We have

\noindent
\[
||\Lambda_{\delta/2}\widehat{u} - \widehat{u}||_{L^1(\widehat{\Omega_{\delta}})} \leqslant C \delta^{1-\varepsilon}
\]

\noindent for some constant 
$C = C(n,\varepsilon,\widehat{\Omega},||\widehat{u}||_{L^{\infty}(\Omega)}) >0 $, 
where $\widehat{\Omega}_{\delta}:= \{ z \in \widehat{\Omega} | d_{\mathbb{C}^n}(z,\partial \widehat{\Omega}) > \delta \}$ 
and for $ z \in \widehat{\Omega}_{\delta}$, $\Lambda_{\delta}\widehat{u}(z) = \frac{1}{\mathbb{B}_{2n} \delta^{2n}}\int_{|z - \zeta| \leqslant \delta}\widehat{u}(\zeta)dV(\zeta)$ 
is the mean volume regularizing function in $\mathbb{C}^n$.

Moreover, if $\Delta \widehat{u}$ has finite mass in $\widehat{\Omega}$ 
then one can get $\delta^2$ instead of $\delta^{1-\varepsilon}$.
    
\end{theorem}

We will also use, in the proof of Lemma~\ref{lema: hom barrier sing}, 
the following classical result from singular Riemannian 
foliations in~\cite[Lemma 1]{W87}, here reformulated to:

\begin{lemma}\label{lema: srf}
    Let $(M,g)$ be a smooth, connected, complete manifold 
    with Riemannian metric $g$ and induced distance 
    function $d$.Fix $f \in C^2(M)$ such that $g(\nabla f, \nabla f) = 1$.
    Suppose $[a,b] \subset f(M) \subset \mathbb{R}$ 
    contains no critical points of $f$.
    Then for any $x \in f^{-1}(a)$ and $y \in f^{-1}(b)$ 
    we have:

\noindent
    \[
    d(x, f^{-1}(b)) = d(f^{-1}(a), y) = \int_{a}^{b} \df = b-a.
    \]
\end{lemma}

\section{Construction of barriers}\label{Sing: barriers}

In this section, we develop barrier functions, which 
allow us to control the modulus of continuity of solutions 
to Dirichlet problems close to the boundary. The barriers 
we will construct are described in the following definition:

\begin{definition}\label{def: barriers sing}
Let $u = u(\Omega, \phi, f)$. 
We call barrier functions for the 
problem $MA(\Omega, \phi, f)$ two functions 
$v,w \in PSH(\Omega) \cap C^{0}(\overline{\Omega})$, 
such that:
\begin{enumerate}
    \item $v(\xi) = \phi(\xi) = -w(\xi)$, $\forall \xi \in \partial \Omega$,
    \item $v(z) \leqslant u(z) \leqslant -w(z), \forall z \in \Omega$,
    %\item $\omega_{v,x}(t),\omega_{w,x}(t) \leqslant C_{x}\max\{ \omega_{\phi}(t^{1/2}) , t^{\gamma} \}$,
     
\end{enumerate}

%\noindent for some constant $C_{x} >0$ that goes 
%to $+\infty$ as $x$ approaches the singular point.
    
\end{definition}

The construction of barriers depends on the behavior around the boundary.
In Lemma~\ref{lema: bdd barrier sing} we construct barriers 
for the case where $\phi \equiv 0$ and $f$ is bounded near 
$\partial \Omega$ (its idea coming from~\cite[Lemma 2.2]{GKZ08}). 
Next, in Lemma~\ref{lema: hom barrier sing} 
we treat the case where $f \equiv 0$
(idea from~\cite[Proposition 4.4]{Cha14}).
The construction of barriers for the general problem $MA(\Omega, \phi, f)$
is postponed to Section~\ref{Sing: Lr}
because one needs to understand the regularity of 
\textbf{solutions} to the particular problems dealt with in this section.

\subsection{Densities bounded near the boundary}

\begin{lemma}\label{lema: bdd barrier sing}

Assume that $\hat{f}$ is bounded near $\partial \Omega$ 
and set $u_0 := u(\Omega, 0 ,\hat{f})$. There exist $b_{\hat{f}}, w \in PSH(\Omega) \cap C^{0,1}(\Omega)$ 
such that:
\begin{enumerate}
    \item $b_{\hat{f}}(\xi) = 0 = -w(\xi), \forall \xi \in \partial \Omega$,
    \item $b_{\hat{f}} \leqslant u_0 \leqslant -w$ in $\Omega$.
\end{enumerate}

\end{lemma}

\begin{proof}
    Since $\hat{f}$ is bounded near $\partial \Omega$, 
    there exists a compact $K\subset \Omega$ such that $0 \leqslant \hat{f} \leqslant M$ 
    on $\Omega \setminus K$.
    Also, there exist $A_1, A_2 >0$ large enough such that 
    ${(dd^c(A_1 \rho))}^n \geqslant M\beta^n \geqslant \hat{f}\beta^n$,
    on $\Omega \setminus K$ and $A_2 \rho \leqslant m \leqslant u_0$ 
    on a neighborhood of $K$, for $m := \underset{\overline{\Omega}}{\min} ~ u_0$.
    Then, by taking the maximum of both constants, we have $A >0$ large 
    enough such that both conditions are satisfied for $b_{\hat{f}} := A\rho$, 
    which is smooth plurisubharmonic on $\Omega$.
    As we also have $b_{\hat{f}} \leqslant u_0$ on 
    $\partial (\Omega \setminus K)$, by the comparison principle 
    we get $b_{\hat{f}} \leqslant u_0$ in $\Omega \setminus K$.
    By construction, we have $b_{\hat{f}} \leqslant u_0$ on 
    a neighborhood of $K$; hence, with the above argument
    we get $b_{\hat{f}} \leqslant u_0$ in $\Omega$ and $b_{\hat{f}}(\partial \Omega) = \{0\}$, 
    thus is a lower-barrier.
    For an upper barrier, take $w \equiv 0$ in $\Omega$ as 
    by the maximum principle $u_0 \leqslant 0$ and it is 
    zero on the boundary.
\end{proof}

\subsection{Zero density}

\begin{lemma}\label{lema: hom barrier sing}
Let $u_{\phi} := u(\Omega, \phi ,0)$. There exists a 
barrier 
$h_{\phi} \in PSH(\Omega) \cap C^{0}(\overline{\Omega})$ 
such that:
\begin{enumerate}
    \item $h_{\phi}(\xi) = \phi(\xi) = -h_{-\phi}(\xi), \forall \xi \in \partial \Omega$,
    \item $h_{\phi} \leqslant u_\phi, \leqslant -h_{-\phi}$ in $\Omega$,
    \item $\omega_{h_{\phi}}(t) \leqslant C \omega_{\phi}(t^{1/2})$.
\end{enumerate}

\end{lemma}

\begin{proof}
We set $h_\phi \in PSH(\Omega) \cap C^{0}(\overline{\Omega})$ such that:

\noindent
\[
\begin{cases}
    \omega_{h_{\phi}}(t) \leqslant C\omega_{\phi}(t^{1/2}) & \text{ in} \quad \Omega,  \\
    h_{\phi} = \phi & \text{ on} \quad \partial \Omega.
\end{cases}
\]

With such $h_{\phi} \in PSH(\Omega)$ we have ${(dd^{c}h_{\phi})}^{n} \geqslant 0 = {(dd^{c}u_{\phi})}^{n}$. 
Hence, by the comparison principle, we get $h_\phi \leqslant u_{\phi}$. 
Similarly, we get $-h_{-\phi} \leqslant -u_{\phi}$; 
hence $h_{-\phi} \geqslant u_{\phi}$.
Now we start the construction of such $h_\phi$:

\begin{itemize}
    \item[\textit{Step 1}:]  \textit{For fixed $\xi \in \partial \Omega$, construct a function $h_\xi$ close to $\xi$.}
\end{itemize}

Fix $\xi \in \partial \Omega$. We want to find a function 
$h_\xi \in PSH(\Omega) \cap C^{0}(\overline{\Omega})$ such that:

\noindent
\[
\begin{cases}
    \omega_{h_{\xi}}(t) \leqslant C\omega_{\phi}(t^{1/2}) & ,  \\
    h_{\xi}(x) \leqslant \phi(x) & \text{ for any } x \in \partial \Omega, \\
    h_{\xi}(\xi) = \phi(\xi) & .
\end{cases}
\]

In this first step, we will construct 
$h_{\xi}$ above just around $\xi$.
Take $B>0$ large enough such that: $g(x) := B\rho(x) - {(d_{\beta}{(x,\xi)})}^2$ is in $PSH(\Omega)$.

Consider $\overline{\omega}_{\phi}$, the minimal concave 
majorant of $\omega_{\phi}$, and define $\chi(t) := -\overline{\omega}_\phi({(-t)}^{1/2})$, 
which is a convex non-decreasing function on $[-d^2,0]$, 
where $d = diam(\Omega)$, the diameter of $\Omega$.

Fix $r>0$ sufficiently small such that $|g(x)| \leqslant d^2$ in $B_{r}(\xi) \cap \Omega$ and define for $x \in B_{r}(\xi) \cap \overline{\Omega}$: $\widetilde{g}(x) = \chi(g(x)) + \phi(\xi) = -\overline{\omega}_{\phi}[{({(d_{\beta}(x,\xi))}^2 - B\rho(x))}^{1/2}] + \phi(\xi)$. 
Note that $\widetilde{g}$ is a continuous psh function on $B_{r}(\xi) \cap \Omega$.

We have $\forall x \in \partial \Omega \cap B_{r}(\xi)$, $\phi(\xi) - \overline{\omega}_{\phi}(d_{\beta}(x,\xi)) \leqslant \phi(x)$.
However, by our construction, $\widetilde{g}(x) = \phi(\xi) - \overline{\omega}_{\phi}(d_{\beta}(x,\xi))$ 
for any $x \in \partial \Omega \cap B_{r}(\xi)$.
Hence, $\widetilde{g}(x) \leqslant \phi(x), \forall x \in \partial \Omega \cap B_{r}(\xi)$ 
and $\widetilde{g}(\xi) = \phi(\xi)$.
Using the subadditivity of $\overline{\omega}_{\phi}$ and also 
the fact that $\forall t,\lambda >0; \omega_{\phi}(\lambda t) \leqslant \overline{\omega}_{\phi}(\lambda t) \leqslant (1+ \lambda) \omega_{\phi}(t)$ 
(see~\cite[Lemma 4.1]{Cha14}) we get:

\noindent
\begin{align*}
    \omega_{\widetilde{g}}(t) & = \underset{d_{\beta}(x,y)\leqslant t}{\sup}|\widetilde{g}(x) -\widetilde{g}(y)| \\
     & \leqslant \underset{d_{\beta}(x,y)\leqslant t}{\sup} \overline{\omega}_{\phi}[{({(d_{\beta}(x,\xi))}^2 -{(d_{\beta}{(y,{\xi})})}^2 - B(\rho(x) - \rho(y)))}^{1/2}] \\
     & \leqslant \underset{d_{\beta}(x,y)\leqslant t}{\sup} \overline{\omega}_{\phi}[{(2d + B_{1})}^{1/2}{(d_{\beta}{(x,y)})}^{1/2}] \\
     & \leqslant (1+{(2d + B_{1})}^{1/2})\underset{d_{\beta}(x,y)\leqslant t}{\sup} \omega_{\phi}[{(d_{\beta}{(x,y)})}^{1/2}] \\
      & \leqslant (1+{(2d + B_{1})}^{1/2})\underset{d_{\beta}(x,y)\leqslant t}{\sup} \omega_{\phi}[{(t)}^{1/2}].
\end{align*}

From the second line to the third, there are the 
following arguments: the first term goes as $d_1^2 - d_2^2 = (d_1 - d_2)(d_1 + d_2)$.
Then, the first difference of distances can 
be bounded by $d_{\beta}(x,y)$ by triangular inequality.
Also, the second term can be bounded by $2d$. 
Finally, the last term comes from the fact 
that, close to the boundary, we have that $|\rho(x) - \rho(y)|$ 
behaves like $d_{\beta}(x,y)$, by Lemma~\ref{lema: srf} up to 
renormalization of the gradient close to the boundary.

\begin{itemize}
        \item[\textit{Step 2}:] \textit{Extending our local $\widetilde{g}$ to all $\overline{\Omega}$ as $h_\xi$.}
\end{itemize}

Recall $\xi \in \partial \Omega$, fix $0 < r_1 < r$ and $\gamma_1 > \frac{d}{r_1}$ 
such that $\forall x \in \overline{\Omega} \cap \partial B_{r_1}(\xi)$:

\noindent
\[
-\gamma_1 \overline{\omega}_{\phi}[{({(d_{\beta}{(x,\xi)})}^2 - B{\rho}(x))}^{1/2}] = \gamma_1 (\widetilde{g}(x) - \phi(\xi)) \leqslant \underset{\partial \Omega}{\inf}~ \phi - \underset{\partial \Omega}{\sup}~ \phi.
\]

Set $\gamma_2 = \underset{\partial \Omega}{\inf}~ \phi$. 
Then:

\noindent
\begin{equation}\label{eq: glueing hom barrier}
    \gamma_1 (\widetilde{g}(x) - \phi(\xi)) + \phi(\xi) \leqslant \gamma_2
\end{equation} 
\noindent for $x \in \partial B_{r_1}(\xi) \cap \overline{\Omega}$. Consider:

\noindent
\[
h_{\xi}(x) := \begin{cases}
    \max\{\gamma_1 (\widetilde{g}(x) - \phi(\xi)) +  \phi(\xi), \gamma_2 \} & \textit{for } x \in B_{r_{1}}(\xi) \cap \overline{\Omega}, \\
    \gamma_2 & \textit{for } x \in \overline{\Omega} \setminus B_{r_{1}}(\xi). \\
\end{cases}
\]

It follows from inequality~\eqref{eq: glueing hom barrier} that $h_\xi$ 
is a psh function on $\Omega$, 
continuous on $\overline{\Omega}$ and such that $h_\xi (x) \leqslant \phi(x), \forall x \in \partial \Omega$ 
because on $\partial \Omega \cap B_{r_{1}}(\xi)$:

\noindent
\[
\gamma_1 (\widetilde{g}(x) - \phi(\xi)) + \phi(\xi) = -\gamma_1 \overline{\omega}_{\phi}[d_{\beta}(x,\xi)] + \phi(\xi) \leqslant -\overline{\omega}_{\phi}[d_{\beta}(x,\xi)] +\phi(\xi) \leqslant \phi(x).
\]

Then finally, we have that $h_\xi$ is a barrier relative to 
the point $\xi$, chosen generically, that is: for each $\xi \in \partial \Omega, \exists h_\xi \in S(\Omega, \phi, 0); h_\xi (\xi) = \phi(\xi)$ 
and $\omega_{h_{\xi}}(t) \leqslant C \omega_{\phi}(t^{1/2})$; 
for $C = (1+{(2d+B_{1})}^{1/2})$ and $S(\Omega,\phi, 0)$ is 
the set of subsolutions for the Dirichlet problem $MA(\Omega, \phi, 0)$.

\begin{itemize}

    \item[\textit{Step 3}:] \textit{Take envelope in $\xi$ to obtain the desired barrier $h_\phi$.}
\end{itemize}

Now set $h_{\phi}(x) = \sup \{h_{\xi}(x) | \xi \in \partial \Omega\}$, which is a sup over a compact family; hence $h_{\phi}^*$ is a psh function.
Note that $0 \leqslant \omega_{h_{\phi}}(t) \leqslant C \omega_{\phi}(t^{1/2})$; then $\omega_{h_{\phi}}(t) \to 0$ as $t \to 0$. This implies $h_{\phi} \in C^{0}(\overline{\Omega})$ 
and $h_{\phi} = h_{\phi}^{*} \in PSH(\Omega)$. 
(Hence $h_{\phi} \in S(\Omega, \phi, 0)$.)
By construction, $h_\phi = \phi$ on $\partial \Omega$.
Thus, we have the desired lower-barrier for $MA(\Omega, \phi, 0)$ 
with modulus of continuity $\omega_{\phi}(t^{1/2})$.
\noindent   
\end{proof}

\section{Regularization of the solution}\label{Sing: reg and mod}

This section provides an appropriate regularization scheme
of the solution $u = u(\Omega, \phi, f)$ to $MA(\Omega, \phi, f)$, 
following ideas from~\cite{GKZ08,DDGHKZ14}.

\subsection{Defining regularization}

First, we define the set $\Omega_{\delta} := \{x \in \Omega \mid d_{\beta}(x, \partial \Omega) > \delta \}$ 
for $0<\delta \leqslant \delta_0$ where $\delta_0 $ is 
fixed such that $\Omega_{\delta_0} \neq \emptyset$.

We consider $\pi: \widetilde{X} \to X$ a resolution of 
singularities of $X$. We will denote objects 
on $\widetilde{X}$ with a tilde. For example, 
$\widetilde{\Omega} = \{ \pi^*\rho := \widetilde{\rho} < 0 \}$.
It is known that $\widetilde{\Omega}$ is a 
Kähler manifold. Fix a 
Hermitian form $\tau$ on $\widetilde{X}$ such that 
$\pi^{*}d \leqslant d_{\tau}$, such as $\tau := \widetilde{C}\pi^{*}\beta + \varepsilon \theta $ 
for some constants $\widetilde{C}>1$ big enough and $0 < \varepsilon$ 
small enough, with some smooth closed (1,1)-form $\theta$ 
Assume for simplicity\footnote{By~\cite[Page 624]{DDGHKZ14} 
and~\cite[Page 2036]{LPT21}, taking $\beta$ and $\tau$ 
not closed will change the exponencial map but not the estimates.} 
that $\tau$ and $\beta$ closed.

Consider for each $\widetilde{x} \in \widetilde{X}$ the 
exponential map $\exp_{\widetilde{x}}: T_{\widetilde{x}}\widetilde{X} \ni \xi \mapsto \exp_{\widetilde{x}}(\xi) \in \widetilde{X}$ 
defined by $\exp _{\widetilde{x}}(\xi)=\sigma(1)$ with $\sigma$ 
being the geodesic such that $\sigma(0) = \widetilde{x}$ 
and initial velocity $\sigma^{\prime}(0)=\xi$. 
Let $\widetilde{u} := \pi^*u$ be the pull-back of the 
solution to $MA(\Omega, \phi, f)$. Define its 
$\delta$-regularization $\eta_{\delta} \widetilde{u}$ as
in~\cite{Dem82} by:

\noindent
$$
\eta_{\delta} \widetilde{u}(\widetilde{x}) = \frac{1}{\delta^{2 n}} \int_{\xi \in T_{\widetilde{x}} \widetilde{X}} \widetilde{u}\left(\exp _{\widetilde{x}}(\xi)\right) \eta\left(\frac{|\xi|_{\tau}^{2}}{\delta^{2}}\right) d V_{\tau}(\xi), \quad \delta>0 \text{ and } \widetilde{x} \in \widetilde{\Omega}_{\delta}.
$$

 Here $\eta$ is a smoothing kernel, $|\xi|_{\tau}^{2}$ 
 stands for $\sum_{i, j=1}^{n} g_{i\bar{j}}(\widetilde{x}) \xi_{i} \bar{\xi}_{j}$, and 
 $d V_{\tau}(\xi)$ is the induced 
 measure $\frac{1}{2^{n} n !}\left(d d^{c}|\xi|_{\tau}^{2}\right)^{n}$. 

It is known that:

\noindent
$$
\exp : T \widetilde{X} \rightarrow \widetilde{X}, \quad T \widetilde{X} \ni(\widetilde{x}, \xi) \mapsto \exp _{\widetilde{x}}(\xi) \in \widetilde{X}, \xi \in T_{\widetilde{x}} \widetilde{X}
$$

\noindent has the following properties:
\begin{enumerate}
    \item $\exp$ is a $C^{\infty}$ mapping;

    \item $\forall \widetilde{x} \in \widetilde{X}, \exp _{\widetilde{x}}(0)=\widetilde{x}$ and $D_{\xi} \exp (0)=\operatorname{Id}_{T_{\widetilde{x}} \widetilde{X}}$.

\end{enumerate}

One can formally extend $\eta_{\delta}\widetilde{u}$ as a function on $\widetilde{\Omega}_{\delta} \times \mathbb{C}$ 
 by putting $U(\widetilde{x}, w):=\eta_{\delta} \widetilde{u}(\widetilde{x})$ 
 for $w \in \mathbb{C}$ with $|w|=\delta$. 
 Coupling the estimate of the hessian of $U(\widetilde{x}, w)$ 
with Kiselman's minimum principle, one gets~\cite[Lemma 1.12]{BD12}:

\begin{lemma}\label{lema: reg dem}
    For a bounded psh function $\widetilde{u}$ on the Kähler manifold $(\widetilde{\Omega} \Subset \widetilde{X}, \tau)$ 
    we let $U(\widetilde{x}, w)$ be its regularization defined as above. 
    Define the Kiselman-Legendre transform at level $c$ by

\noindent
$$
\widetilde{u}_{c, \delta}(\widetilde{x}) :=\inf _{0 \leq t \leq \delta}\left[U(\widetilde{x}, t)+K t^{2}-K \delta^{2}-c \log (t / \delta)\right]
$$

\noindent for $\widetilde{x} \in \widetilde{\Omega}_{\delta}$ 
and some positive constant $K$ depending on 
the curvature of $(\overline{\widetilde{\Omega}}, \tau)$ such that the function $U(\widetilde{x}, t)+K t^{2}$ 
is increasing for $t \in (0,\delta_1)$ for 
some $0 < \delta_1 \leqslant \delta_0$ small enough. Also one 
has the following estimate for the complex hessian:

\noindent
$$
d d^{c} \widetilde{u}_{c, \delta} \geqslant -\left(Ac + K \delta\right) \tau
$$

\noindent where $A$ is a lower bound of the negative part of the bisectional curvature of $(\overline{\widetilde{\Omega}},\tau)$.
\end{lemma}

\subsection{Correcting the positivity}

Before extending to the boundary as in~\cite{GKZ08}, we need to 
correct the positivity of our regularizing function. 
The construction above creates a well behaved
regularizing function on compact subsets. However, we 
require plurisubharmonicity to use the stability estimate.

As in~\cite[Section 3.2.4]{GGZ23b}, since $X$ has an isolated singularity, 
we get the following argument: The exceptional divisor $E$ 
of the resolution has the property that there exist 
positive rational numbers $(b_i)_{i \in I}$ 
such that $-\sum_{i \in I}b_i E_i$ is $\pi$-ample, 
where each $E_i$ is an irreducible component of $E$. 
Then, on each $\mathcal{O}_{\widetilde{X}}(E_i)$ 
pick a section $s_i$ cutting out $E_i$ and choose 
an appropriate smooth Hermitian metric $h_i$ such that

\noindent
\[
\rho^{\prime} := B_1(\pi^*(B_0\rho) + \sum_{i \in I}b_i \log|s_i|_{h_i}^2)
\]

\noindent is strictly psh in $\widetilde{\Omega}$ for 
some $B_0 > 1$ big enough and choose $B_1 > 0$ such that $dd^{c}\rho^{\prime} \geqslant \tau$.
Assume without loss of generality $\widetilde{u} \geqslant 0 $. 
Now, assuming\footnote{One should keep in mind that Section~\ref{Sing: barriers} already garanties the existence of such barriers, and a control in their modululs of continuity, for some specific Dirichlet problems.} 
the existence of barriers as in 
Definition~\ref{def: barriers sing}, we fix:

\noindent
\begin{equation}\label{eq: Kiselman reg func}
    \widetilde{u}_\delta := \widetilde{u}_{c ,\delta}; \quad  Ac:=   H(\delta) - K\delta
\end{equation}

\noindent for $H: \mathbb{R}^{+} \to \mathbb{R}^{+}$ 
continuous, increasing function such that $\omega_{H}(\delta) \gtrsim \omega_{v}(\delta) \gtrsim O(\delta)$ 
and $H(0) = 0$, for $v$ a barrier from 
Definition~\ref{def: barriers sing} 
and $A,K$ from Lemma~\ref{lema: reg dem}. Take $\delta < \delta_2 \leqslant \delta_1$ 
for $\delta_2$ small enough such that the $c$ above is 
positive. We have control
on its loss of positivity, hence:

\noindent
\[
\widehat{u}_\delta := \widetilde{u}_\delta + (Ac + K\delta)\rho^{\prime} = \widetilde{u}_\delta + H(\delta)\rho^{\prime}.
\]

We have $dd^{c}\widehat{u}_{\delta} = dd^{c}\widetilde{u}_\delta + (Ac + K\delta)dd^{c}\rho^{\prime}$. By the 
definition of $\rho^{\prime}$ and Lemma~\ref{lema: reg dem} 
we get $\widehat{u}_{\delta} \in PSH(\widetilde{\Omega}_{\delta})$.
Now, we set $\check{u}_{\delta} := \pi_{*}\widehat{u}_{\delta}$,
that is for any $x \in \pi^{-1}(\Omega_{\delta}); \check{u}_{\delta}(\pi(x)) = \widehat{u}_{\delta}(x)$,
which is well defined on all $\Omega_{\delta}$ by the 
following arguments: since on $E = \pi^{-1}(0)$ the $\widehat{u}_{\delta}$ 
is constantly $-\infty$ by construction, and on the outside of $E$ we 
have that $\pi$ is an isomorphism. Moreover $\pi^{-1}(\Omega_{\delta}) \Subset \widetilde{\Omega}_{\delta}$, 
because $d \leqslant d_{\tau}$ by definition of $\tau$.

\subsection{Extending regularization to all \texorpdfstring{$\Omega$}{Omega}.}

Here, we will assume the existence of barriers as in 
Defnition~\ref{def: barriers sing} 
to control $u$ $\delta-$close to the boundary. 
The extension will be, for some constant $C>0$:

\noindent
\[
\overline{u}_\delta = \begin{cases}
    \max\{\check{u}_\delta, u+ 4CH(\delta) \} &\text{ in } \Omega_{2\delta} , \\
u + 4CH(\delta) &\text{ in }\Omega \setminus \Omega_{2\delta}.
\end{cases}
\]

Notice that because $\check{u}_\delta$ goes to $-\infty$ 
at $X_{\text{sing}}$, then close to the singularity 
(inside some ball $B_{\mu}(X_{\text{sing}}),\text{ for }$ $\mu>0$ small enough) $u + 4CH(\delta)$ 
will eventually be the greater term inside the max above, making $\overline{u}_{\delta}$ 
continuous and psh in $\Omega$.
For the gluing to be in $PSH(\Omega)$ we need that 
$\check{u}_\delta(x) \leqslant u(x) + 4C H(\delta), \forall x \in \partial \Omega_{2\delta}$. 
The barriers will give us this last inequality through the following arguments:

\subsubsection{Behaviour at the boundary}

Fix some $\lambda_0 > 0 $ small enough such that $\overline{B_{\lambda_0}(X_{\text{sing}})} \Subset \Omega_{\delta_0}$.

\begin{lemma}\label{lema: bound rel bound sing}
Assume $MA(\Omega, \phi, f)$ admits barrier functions $v,w$.
The solution $u = u(\Omega, \phi, f)$ satisfies, for some constant $C^{\prime}>0$:

\noindent
\[
|u(x) - u(\xi)| \leqslant C^{\prime} \omega_{v}(d(x , \xi)), \forall x \in \Omega \setminus B_{\lambda_0}(X_{\text{sing}}); \forall \xi \in \partial \Omega.
\]

\end{lemma}

\begin{proof}
From Definition~\ref{def: barriers sing}:

\noindent
\[
v(x) - v(\xi) \leqslant u(x) - \phi(\xi) \leqslant - (w(x) - w(\xi)), \forall x \in \Omega \setminus B_{\lambda_0}(X_{\text{sing}}); \forall \xi \in \partial \Omega.
\]

By the modulus of continuity of the barriers 
we get, for some constant $C^{\prime}>0$:

\noindent
\[
|u(x) - u(\xi)| \leqslant C^{\prime} \omega_{v}(d(x , \xi)), \forall x \in \Omega \setminus B_{\lambda_0}(X_{\text{sing}}); \forall \xi \in \partial \Omega.
\]
\noindent
\end{proof}

\subsubsection{Behaviour near the boundary}

\begin{lemma}\label{lema: bound close bound sing}
    Assume $MA(\Omega, \phi, f)$ admits barrier functions $v,w$. 
    Take $r_0>r>0$ where $\overline{B_{r_0}(\xi)}\cap \Omega \subset \Omega \setminus B_{\lambda_0}(X_{\text{sing}}); \forall \xi \in \partial \Omega$. For any $\xi \in \partial \Omega$, 
Then the solution $u$ satisfies the following property:

\noindent
\[
|u(x_1 ) - u(x_2 )| \leqslant 2C^{\prime} \omega_{v}(r)
\]

\noindent for some constant $C^{\prime}>0$ and 
$\forall x_1,x_2 \in \overline{B_{r}(\xi)} \cap \Omega$.

\end{lemma}

\begin{proof}
    Fix $r>0$ and an arbitrary $\xi \in \partial \Omega$. 
    Take any two points 
    $x_1 , x_2 \in \overline{B_{r}(\xi)} \cap \Omega$. 
    Using the triangular inequality we get:
\begin{align*}
    |u(x_1 ) - u(x_2 )| & \leqslant |u(x_1 ) - u(\xi)| + |u(\xi) - u(x_2 )| \\
     & \leqslant C \omega_{v}(d(x_1 , \xi)) + C \omega_{v}(d(\xi , x_2 )) \\
     & \leqslant 2C \omega_{v}(r).
\end{align*}

The second line comes from Lemma~\ref{lema: bound rel bound sing}.
\noindent
\end{proof}

\subsubsection{Extension of regularization.}

By construction, the value of 
$u_{\delta} := \pi_*\widetilde{u}_{\delta}$ is under 
control relative to the supremum on a ball 
of radius $\delta$ for the points away from the 
singularity, where the resolution is an isomorphism. 
To prove the inequality necessary for the glueing process, 
assuming $MA(\Omega, \phi, f)$ admits barrier function $v,w$, 
we use an argument by contradiction:

\begin{proof}
Assume by contradiction that 
$\exists x_0 \in \partial \Omega_{2\delta}$ such 
that $\check{u}_\delta(x_0) > u(x_0) + 4CH(\delta)$. 
since $u_{\delta}(x_0) \leqslant \underset{\overline{B_{\delta}(x_{0})}}{\max}~ u$, 
then $\exists~ x^* \in \overline{B_{\delta}(x_0)}$ 
such that $u(x^*) + H(\delta)\pi_*\rho^{\prime}(x_0)> u(x_0) + 4CH(\delta)$.
The points on $\partial \Omega_{2\delta} $ 
have distance $2\delta$ to $\partial \Omega$, then take 
$\xi_0 \in \partial \Omega$ such that 
$d(\xi_0,x_0) = 2\delta$. We use the 
Lemma~\ref{lema: bound close bound sing} for 
$r = 4\delta$, $x_1 = x_0$ and $x_2 = x^*$, for $\delta < \min\{ r_0/2, \delta_2 \}$. 
Then:

\noindent
\[
    |u(x_0) - u(x^*)| \leqslant 2C^{\prime}\omega_{v}(4\delta).
\]

By assumption $|u(x^*) - u(x_0)| > 4CH(\delta) - H(\delta)\pi_*\rho^{\prime}(x_0) \geqslant 8C^{\prime}H(\delta)$. 
For $C \geqslant 2C^{\prime} + \underset{\partial{\Omega_{\delta}}}{\inf}\pi_*\rho^{\prime} + \underset{\Omega}{\sup}~\pi_*\rho^{\prime} > 0$.
Then we get: $2C^{\prime}(\omega_{v}(4\delta) - 4H(\delta)) > 0$, which is a
contradiction for any $ \omega_{v}(\delta) \gtrsim \delta$, 
as $\omega_{H}(\delta) \gtrsim \omega_{v}(\delta)$.
\noindent
\end{proof}

Hence, the global extension 
$\overline{u}_\delta$ is continuous and psh 
as we wanted to construct.

\section{\texorpdfstring{$L^1$}{L1} Estimate and general case}\label{Sing: Lr}

We will use the local theory developed in~\cite{GKZ08},
and then refined in~\cite{BKPZ16, Cha15}, 
to obtain the $L^1$ estimate of the regularizing 
function using a Laplacian estimate. With this we will
obtain the desired regularity of solutions to the 
Dirichlet problems we already have barriers for, 
from Section~\ref{Sing: barriers}, 
and then treat the general case.

\subsection{\texorpdfstring{$L^1$}{L1} Laplacian estimate}

Here, we calculate the estimate that dictates the behavior
away from the boundary, derived from the local 
theory using a Laplacian estimate. 
First, we compare the local regularizing 
function with the one constructed in Section~\ref{Sing: reg and mod}:

\begin{lemma}\label{lema: ctrl term 2}
    Let $\varepsilon>0$ be small enough, and $u \in PSH(\Omega) \cap C^{0}(\overline{\Omega})$
    with $(dd^{c}u)^{n} = f \beta^{n}$ in $\Omega$. Then 
    for any sufficiently small $\delta >0$ and chart on 
    the resolution $(U \Subset \widetilde{\Omega_{\delta/2}},\psi)$ 
    we have:

\noindent
\[
\eta_{\delta/2}\widetilde{u} \leqslant C_1 \Lambda_{\delta/2}\widetilde{u} + C_2 \delta^{2}
\]

\noindent for some constants $C_1, C_2>0$ and $\Lambda_{\delta}\widetilde{u}(\widetilde{x}) := \Lambda_{\delta} (\widetilde{u} \circ \psi^{-1})(\psi(\widetilde{x}))$ for any $\widetilde{x} \in U$.

Moreover, we have
\(
||\eta_{\delta}\widetilde{u} - \widetilde{u}||_{L^{1}(\widetilde{\Omega_{2\delta}})} \leqslant \widetilde{C}\delta^{1-\varepsilon},
\) for some constant $\widetilde{C}>0$. Also, if $\Delta \widetilde{u}$ has finite mass in $\widetilde{\Omega}$ 
then one can get $\delta^2$ instead of $\delta^{1-\varepsilon}$.

\end{lemma}

\begin{proof}
For a fixed chart $(U \Subset \widetilde{\Omega_{\delta/2}},\psi)$ we have from the proof of~\cite[Lemma 2.3]{DDGHKZ14}:

\noindent
\[
\eta_{\delta/2}\widetilde{u}(\widetilde{x}) = \frac{1}{(\delta/2)^{2n}} \int_{\widetilde{y} \in \widetilde{X}} \widetilde{u}(\widetilde{y}) \eta\left( \frac{|\log_{\widetilde{x}}\widetilde{y}|}{(\delta/2)^{2}} \right) dV_{\tau}(\log_{\widetilde{x}}\widetilde{y})
\]

\noindent where $\widetilde{y} \mapsto \xi = \log_{\widetilde{x}}\widetilde{y}$ 
is the inverse function 
of $\xi \mapsto \widetilde{y} = \exp_{\widetilde{x}}(\xi)$. 
By the proof of~\cite[Lemma 2.4]{DDGHKZ14} we have:

\noindent
\[
dV_{\tau}(\log_{\widetilde{x}}\widetilde{y}) = \underset{j=1}{\overset{n}{\bigwedge}}\frac{i}{2}(dz_{j} - dw_{j}) \wedge (d\bar{z}_{j} - d\bar{w}_{j}) + O(d_{\tau}{({\widetilde{y}},\widetilde{x})}^{2}),
\]

\noindent where $(w,z) \mapsto (\widetilde{x}, \log_{\widetilde{x}}(\widetilde{y}))$  
represent local coordinates 
on a neighborhood of the zero section in $TU$, while $ (\widetilde{x}, \log_{\widetilde{x}}(\widetilde{y})) \mapsto (\widetilde{x}, \widetilde{y})$ 
is a diffeomorphism from that neighborhood onto the 
diagonal in $U\times U$. The $O(\cdot)$ term 
depends only on the curvature. 
By definition, coupled with the above equality:

\noindent
\[
\eta_{\delta/2}\widetilde{u}(\widetilde{y}) - \Lambda_{\delta/2}\widetilde{u}(\widetilde{y}) \leqslant
C_1 \Lambda_{\delta/2}\widetilde{u}(\widetilde{y}) + C_{2} \delta^{2},
\]

\noindent for any $\widetilde{y} \in U$ and some constants $C_1,C_2 >0$. Moreover, taking $\sup \eta = 1$ one can choose $C_1 = 0$, consequently applying Theorem~\ref{teo: local estimate sing}, one gets:

\noindent
\[
||\eta_{\delta/2}\widetilde{u} - \widetilde{u}||_{L^{1}(U)} \leqslant C_{3}\delta^{1-\varepsilon},
\]

\noindent for some constant $C_3 >0$ for the chart $(U, \psi)$.
Then, by compactness of $\overline{\widetilde{\Omega_{\delta}}}$,
the desired estimate is achieved for some constant $\widetilde{C}>0$:

\noindent
\[
||\eta_{\delta}\widetilde{u} - \widetilde{u}||_{L^{1}(\widetilde{\Omega_{2\delta}})} \leqslant \widetilde{C}\delta^{1-\varepsilon}.
\]
\end{proof}

\begin{remark}
The proof doesn't involve the boundary values of $u$. 
This is clear by the proof of~\cite[Theorems 3.3, 3.4 and Lemma 3.5]{BKPZ16},
nor it involves the barrier functions.

\end{remark}

\subsection{Main estimate}

Theorem~\ref{teo: final estimate} bellow gives 
us, provided we have appropriate barrier functions, our 
main estimate~\ref{teo: main result} 
and also~\ref{teo: holder sing}:

\begin{theorem}\label{teo: final estimate}
Assume $MA(\Omega, \phi,f)$ admits barriers $v,w$. The unique solution $u = u(\Omega, \phi, f)$ which belongs to 
 $PSH(\Omega) \cap C^{0}(\overline{\Omega})$ to $MA(\Omega,\phi,f)$, satisfies for any $x \in \Omega$:
\[
\omega_{u,x}(t) \leqslant C_{x} \omega_{H}(t)
\]
for some constants $C_{x} >0$ such that $C_{x} \to +\infty$ 
as $x \to X_{\text{sing}}$, $0 \leqslant \gamma < \frac{1}{(nq+1)}$ 
and $\omega_{H}(t) \gtrsim \max\{ \omega_{\phi}(t^{1/2}), t^{\gamma} \}$ 
or $\max\{ \omega_{\phi}(t^{1/2}), t^{2\gamma}\}$ if 
$\Delta \widetilde{u}$ has finite mass in $\Omega$.
\end{theorem}

\begin{proof}

Fix $\gamma < \frac{1}{(nq+1)}$ and $\varepsilon >0$ sufficiently small that $\gamma^{\prime} := \frac{\gamma}{(1- \varepsilon)} < \frac{1}{(nq+1)}$. 
Applying Theorem~\ref{teo: stability sing} for $\gamma^{\prime}$:

\noindent
\begin{align*}
    \underset{x \in \Omega}{\sup}(\overline{u}_{\delta}(x) - u(x) -4CH(\delta)) \leqslant &  \underset{x \in \partial\Omega}{\sup}(\overline{u}_{\delta}(x) - u(x) -4CH(\delta))^* + \overline{C}||  \overline{u}_{\delta} - u -4CH(\delta) ||_{L^{1}(\Omega)}^{\gamma^{\prime}}  \\ 
    \leqslant & \overline{C}||  \overline{u}_{\delta} - u - 4CH(\delta) ||_{L^{1}(\Omega_{2\delta} \setminus B_{\mu}(X_{\text{sing}}))}^{\gamma^{\prime}} \\
    \leqslant & \overline{C}||  \check{u}_{\delta} - u - 4CH(\delta) ||_{L^{1}(\Omega_{2\delta} \setminus B_{\mu}(X_{\text{sing}}))}^{\gamma^{\prime}} \\
    \leqslant & \overline{C} ||  \pi_*\widetilde{u}_{\delta} - u ||_{L^{1}(\Omega_{2\delta} \setminus B_{\mu}(X_{\text{sing}}))}^{\gamma^{\prime}}  \\
    \leqslant & \overline{C} ||  \pi_*\eta_{\delta}\widetilde{u} - u ||_{L^{1}(\Omega_{2\delta} \setminus B_{\mu}(X_{\text{sing}}))}^{\gamma^{\prime}} \\
    \leqslant & \overline{C} ||  \eta_{\delta}\widetilde{u} - \widetilde{u} ||_{L^{1}(\widetilde{\Omega_{2\delta}} \setminus \pi^{-1}[B_{\mu}(X_{\text{sing}})])}^{\gamma^{\prime}} \\
    \leqslant &  \overline{C}(\widetilde{C}^{\gamma^{\prime}}\delta^{(1-\varepsilon)\gamma^{\prime}}) \\
    \leqslant & C_0\delta^{\gamma}
\end{align*}

Close to $\partial \Omega$ and $X_{\text{sing}}$ the terms are controlled
by construction.
From lines 3 to 4, we use the fact that $H(\delta)\rho^{\prime} - 4CH(\delta) \leqslant 0$.
Remember that outside the singularity/divisors $\pi$ is an isomorphism.
The second to last passage is achieved by applying Lemma~\ref{lema: ctrl term 2}.
Hence, we get:
 \begin{align*}
     C_0 \delta^{\gamma} + 4CH(\delta) \geqslant & \underset{x \in \Omega}{\sup}(\overline{u}_{\delta}(x) - u(x))  \\
     \geqslant & \overline{u}_{\delta}(x) - u(x), \quad \forall x \in \Omega \\
     \geqslant & \pi_*\widetilde{u}_{\delta}(x) - u(x) + H(\delta)\pi_*\rho^{\prime}(x), \quad \forall x \in \Omega_{\delta} \\
     \geqslant & \widetilde{u}_{\delta}(\widetilde{x}) - \widetilde{u}(\widetilde{x}) + H(\delta)M\log(d_{\tau}(\widetilde{x},E)), \quad \forall \widetilde{x} \in \widetilde{\Omega_{\delta}}
 \end{align*}

As $\rho^{\prime}(\widetilde{x}) \geqslant M\log(d_{\tau}(\widetilde{x},E))$,
for any $\widetilde{x} \in \widetilde{\Omega}$ and some constant $M>0$.
Then, for $S(\widetilde{\lambda}) := M (-\log(d_{\tau}(\widetilde{x},E)))$:\footnote{The notation $\widetilde{\lambda}(\widetilde{x})$ in $\widetilde{X}$ is the analogous of $\lambda(x)$ in $X$.}

\noindent
\begin{equation}\label{eq: estimate Kiselman}
\widetilde{u}_{\delta}(\widetilde{x}) - \widetilde{u}(\widetilde{x}) \leqslant (C_1 +  S(\widetilde{\lambda}))H(\delta), \quad \forall \widetilde{x} \in \widetilde{\Omega_{\delta}}
\end{equation}

Following the proof of~\cite[Theorem D]{DDGHKZ14}, 
we get a uniform lower bound on the parameter $t = t(\widetilde{x})$
that realizes the infimum in the Kiselman-Legendre 
transform for $\widetilde{u}_{\delta}$ at a 
fixed $\widetilde{x} \in \widetilde{\Omega_{\delta}}$:

\noindent
\[
\widetilde{u}_{\delta}(\widetilde{x}) - \widetilde{u}(\widetilde{x}) = \eta_{t}\widetilde{u}(\widetilde{x}) + Kt^2 - \widetilde{u}(\widetilde{x}) - K\delta^2 -c  \log(t/\delta) \leqslant (C_1 + S(\widetilde{\lambda}))H(\delta)
\]

As $t \mapsto \eta_t \widetilde{u} + Kt^2$ is increasing, 
we have $\eta_t \widetilde{u} + Kt^2 - \widetilde{u} \geqslant 0$;
thus:

\noindent
\[
c \log (t/\delta) \geqslant -(C_2 + S(\widetilde{\lambda}))H(\delta).
\]

Since $c = A^{-1}H(\delta) - A^{-1}K\delta  = H(\delta)(A^{-1} - A^{-1}K\frac{\delta}{H(\delta)})$,
we get the bound $c \geqslant \frac{A^{-1}H(\delta)}{2}$,
by choosing  $\delta \leqslant \delta_3 := \min\{r_0/2,  \delta_0,\delta_1, \delta_2, \varpi^{-1}(1/2K) \}$ 
for $\varpi(\delta) = \frac{\delta}{H(\delta)}$, because $\omega_{H}(\delta) \gtrsim \delta$.
Then:

\noindent
\[
\delta \geqslant t(\widetilde{x}) \geqslant \delta \kappa(\widetilde{x}),
\]

\noindent where

\noindent
\begin{equation}\label{eq: kappa}
    \kappa(\widetilde{x}) := e^{-2A(C_2 + S(\widetilde{\lambda}))} = e^{-2AC^{\prime}_2}\cdot (\widetilde{\lambda}(\widetilde{x}))^{2AM}
\end{equation}

Finally, using that $t \mapsto \eta_t \widetilde{u} + Kt^2$
is increasing, $t(\widetilde{x}) \geqslant \delta \kappa(\widetilde{x})$ 
and the inequality \eqref{eq: estimate Kiselman}, for every $\widetilde{x} \in \widetilde{\Omega_{\delta}}$ 
we get:

\noindent
\begin{equation}\label{eq: conv est intermediate}
    \eta_{\delta \kappa(\widetilde{x})}\widetilde{u}(\widetilde{x}) - \widetilde{u}(\widetilde{x}) - K\delta^{2} \leqslant \widetilde{u}_{\delta}(\widetilde{x}) - \widetilde{u}(\widetilde{x}) \leqslant (C_1 + S(\widetilde{\lambda}))H(\delta)
\end{equation}

\noindent which leaves us with:

\noindent
\begin{equation}\label{eq: final estimate}
        \eta_{\delta}\widetilde{u}(\widetilde{x}) - \widetilde{u}(\widetilde{x}) \leqslant (C_3 + S(\widetilde{\lambda}))H(\delta/\kappa(\widetilde{\lambda})) \leqslant C_{\widetilde{\lambda}}\omega_{H}(\delta), \quad \forall \widetilde{x} \in \widetilde{\Omega_{\delta}}
\end{equation}

\noindent by the subadditivity of $\omega_{H}$ we get: $C_{\widetilde{\lambda}} = \frac{(C_3 + S(\widetilde{\lambda}))}{\kappa(\widetilde{\lambda})} = \frac{[C_3 + M^{\prime}(-\log(\widetilde{\lambda}(\widetilde{x})))]}{(\widetilde{\lambda}(\widetilde{x}))^{2AM}}$.

Fix $\widetilde{x} \in \widetilde{\Omega} \setminus E$, 
then there exists $\delta^{\prime} \leqslant \min\{ 1/2, \delta_3, \lambda(x)/2, d(x,\partial \Omega )/2\}$.
Now fix any $0 < \delta < \delta^{\prime}$, hence $\widetilde{x} \in \widetilde{\Omega_{\delta}}$.
Applying~\cite[Theorem 3.4]{Ze20} for $\overline{\widetilde{\Omega_{\delta}}}$, $\widetilde{u}$
and inequality \eqref{eq: final estimate}:

\noindent
\begin{equation}\label{eq: holder on resolution}
    |\widetilde{u}(\widetilde{x}) - \widetilde{u}(\widetilde{y})| \leqslant D_{0}(C_{\widetilde{\lambda}(\widetilde{x})} + C_{\widetilde{\lambda}(\widetilde{y})})\omega_{H}(d_{\tau}(\widetilde{x},\widetilde{y})),
\end{equation}

\noindent for any $\widetilde{y} \in \widetilde{\Omega_{\delta}}$ and some constant $D_{0} > 0$.
Then, one can pass the right-hand side of \eqref{eq: holder on resolution} to $\Omega_{\delta}$ and substitute 
$\widetilde{\lambda}$ by $\lambda$, remember that $\tau := \widetilde{C}\pi^{*}\beta + \varepsilon \theta $ 
for some constants $\widetilde{C}>1$ big enough and $0 < \varepsilon$ 
small enough, with some smooth closed (1,1)-form $\theta$ 
so that $\pi^{*}d \leqslant d_{\tau}$.
Hence:

\noindent
\begin{equation}\label{eq: holder on middle}
    |u(x) - u(y)| \leqslant (C_{\lambda(x)} + C_{\lambda(y)})\omega_{H}(d_{\tau}(\widetilde{x},\widetilde{y})),
\end{equation}

\noindent for any $\widetilde{y} \in \widetilde{\Omega_{\delta}}$ and some constant $C_{\lambda} >0 $ that goes to $+ \infty$
as $\lambda \to X_{\text{sing}}$.
Lastly, notice that we have the comparison $\tau_{\widetilde{x}} \leqslant K_{\lambda(x)}\pi^{*}\beta_{\widetilde{x}}$
for some constant $K_{\lambda}>0$ that goes to $+\infty$ 
as $\lambda \to X_{\text{sing}}$.
Now, by the subadditivity of $\omega_{H}$ we get:

\noindent
\begin{equation}\label{eq: holder on variety}
    |u(x) - u(y)| \leqslant (C^{\prime}_{\lambda(x)})\omega_{H}(d(x,y)) \leqslant (C^{\prime}_{\lambda(x)})\omega_{H}(\delta)
\end{equation}

\noindent for any $y \in B_{\delta}(x)$ and some constant $C^{\prime}_{\lambda} >0$ that 
goes to $+\infty$ as $\lambda \to X_{\text{sing}}$.
\noindent
\end{proof}

\begin{remark}
    One can notice that the proof of~\cite[Theorem 3.4]{Ze20}
    follows even if $\overline{\widetilde{\Omega_{\delta}}}$ have a 
    boundary because $\widetilde{u}$ is defined on 
    all $\widetilde{\Omega}$.
\end{remark}
\noindent

\begin{corollary}\label{cor: holder intermediate sing}

    Let $\hat{f} \in L^p(\Omega, \beta^n), p > 1$ with $\hat{f}$ 
    is bounded near $\partial \Omega$ and $0 \leqslant \gamma < \frac{1}{(nq + 1)}$. 
    Then $u(\Omega,\phi, 0), u(\Omega,0, \hat{f})$ have 
    modulus of continuity $C_{x}\max\{ \omega_{\phi}(t^{1/2}), t^{\gamma} \}$ 
    and $C_{x}t^{2\gamma}$ respectively, for $x \in \overline{\Omega}$ 
    and some constant $C_{x}>0$ that goes to $+\infty$ as $x \to X_{\text{sing}}$.

    \end{corollary}

    \begin{proof}
    Since we already have barriers for the problems 
    $MA(\Omega, \phi, 0)$ and $MA(\Omega, 0, \hat{f})$, we
    can choose $H$ to have modulus of continuity $\max\{ \omega_{\phi}(t^{1/2}), t^{\gamma} \}$, 
    since their respective barriers have better regularity then that.
    Then from Theorem~\ref{teo: final estimate} 
    the solution $u(\Omega, \phi, 0)$ will have modulus of 
    continuity $C_{x}\max\{ \omega_{\phi}(t^{1/2}), t^{\gamma} \}$.
    Similarly, we get that $u(\Omega, 0, \hat{f})$ will have
    modulus of continuity $C_{x}t^{\gamma}$, because the 
    boundary data is zero. One can improve the 
    regularity to $C_{x}t^{2\gamma}$ by first noticing that the 
    barriers for $MA(\Omega, 0, \hat{f})$ in 
    Lemma~\ref{lema: bdd barrier sing} have 
    Laplacians with uniformly bounded mass in $\Omega$, 
    using Chern-Levine-Nirenberg inequalities\footnote{It is known that for isolated singularities inside the domain the result follows in the exact same way as the local theory. For that one can follow directly the proof in~\cite[Theorem 3.9]{GZ17}.}
    (\cite{CLN69},~\cite[Theoreme 2.2]{Dem85}, see~\cite[Theorem 3.9]{GZ17} for a more modern presentation) 
    coupled with them being bounded and psh on 
    a neighborhood of $\overline{\Omega}$. Now, if $\beta$ is Kähler 
    then by the 
    comparison principle, we get that the solution 
    has Laplacian with uniformly bounded mass. 
    If $\beta$ is not then it suffices to notice that in a 
    neighborhood of $\Omega$ it is bigger/smaller then a Kähler form, 
    such as $dd^{c}(A\rho)$ for some constant $A>0$ big/small enough. 
    Lastly, as this argument is the same in $\Omega$ or $\widetilde{\Omega}$, 
    therefore we get the regularity $t^{2\gamma}$
    in Lemma~\ref{lema: ctrl term 2}.
    \noindent
    \end{proof}

\subsection{The general case}

Now we construct the barrier functions
for the solution of the general problem $u := u(\Omega, \phi, f)$:

\begin{proposition}\label{prop: barriers sing}

    Let $u = u(\Omega, \phi, f)$. 
    Then there exists two barrier functions 
    $v,w \in PSH(\Omega) \cap C^{0}(\overline{\Omega})$, 
    such that:
    \begin{enumerate}
        \item $v(\xi) = \phi(\xi) = -w(\xi)$, $\forall \xi \in \partial \Omega$,
        \item $v(z) \leqslant u(z) \leqslant -w(z), \forall z \in \Omega$,
        \item $\omega_{v,x}(t),\omega_{w,x}(t) \leqslant C_{x}\max\{ \omega_{\phi}(t^{1/2}) , t^{\gamma} \}$
         
    \end{enumerate}
    
    \noindent for some constant $C_{x} >0$ that goes 
    to $+\infty$ as $x$ approaches the singular point.

\end{proposition}

\begin{proof}
The upper-barrier will be $w = u(\Omega, -\phi, 0)$. 
Since $-w = u(\Omega, \phi, 0)$, $-w = \phi$ in $\partial \Omega$ 
and $0 = (dd^c (-w))^{n} \leqslant {(dd^c u)}^n$, by the comparison principle 
we get $-w \geqslant u$ in $\Omega$ and by the Corollary~\ref{cor: holder intermediate sing} 
we know the modulus of continuity of $w$.

For the lower-barrier, we take a bigger pseudoconvex domain $\Omega \Subset \widehat{\Omega} \Subset X$ 
and extend trivially by zero the density $f$ to $\widehat{\Omega}$, name it $\widehat{f}$. 
Note that $\widehat{f}$ is bounded near $\partial \widehat{\Omega}$.
The lower-barrier will be 
$v := u(\widehat{\Omega}, 0 , \widehat{f})|_{\Omega} + u(\Omega, \phi - u(\widehat{\Omega}, 0 , \widehat{f})|_{\partial\Omega}, 0)$.
By construction $v|_{\partial \Omega} = \phi$ and $(dd^{c}v)^{n} \geqslant f\beta^{n} + 0$ in $\Omega$, 
by the comparison principle $v\leqslant u$ in $\Omega$ and by 
Corollary~\ref{cor: holder intermediate sing} 
we know the modulus of continuity of $v$.
\noindent
\end{proof}

Now that we have barriers for the general case 
$MA(\Omega, \phi, f)$ we can prove~\ref{teo: main result}
and~\ref{teo: holder sing} :

\begin{proof}
    By choosing $\omega_{H}(t) = \max\{ \omega_{\phi}(t^{1/2}), t^{\gamma} \}$ 
    on Theorem~\ref{teo: final estimate} 
    one proves~\ref{teo: main result}, 
    this choice  of $H$ is possible because of the modulus of 
    continuity of the barrier functions constructed in 
    Proposition~\ref{prop: barriers sing}.
    \ref{teo: holder sing} 
    follows directly as for it $\omega_{\phi}(t) = t^{\alpha}$.  
\end{proof}

\bibliographystyle{alpha}
\bibliography{ref.bib}

\vspace{4pt}
\hrule

\end{document}